\documentclass[11pt,reqno]{amsart}
\usepackage[utf8]{inputenc}
\usepackage[T1]{fontenc}
\usepackage{lmodern}
\usepackage{amsfonts,amsthm,amsmath,amssymb,mathtools}
\usepackage{graphicx}
\usepackage[dvipsnames,table]{xcolor}   
\usepackage{enumerate}
\usepackage{hyperref}
\usepackage{color}
\usepackage{tikz-cd}
\usepackage{subcaption}
\usepackage[margin=1in]{geometry}
\usepackage[
    maxbibnames=99,
    backend=biber,
    style=alphabetic,
    sorting=nyt,
    giveninits=true
]{biblatex}
\DeclareFieldFormat{pages}{#1}
\renewbibmacro{in:}{
  \ifentrytype{article}
    {}
    {\bibstring{in}
     \printunit{\intitlepunct}}}
\DeclareFieldFormat
[article,inbook,incollection,inproceedings,patent,thesis,unpublished]
  {title}{\mkbibemph{#1}}
\DeclareFieldFormat{journaltitle}{#1\isdot}
\DeclareFieldFormat[article]{volume}{\mkbibbold{#1}}
\DeclareFieldFormat[article]{number}{\bibstring{number}\addnbspace #1}

\renewbibmacro*{journal+issuetitle}{
  \usebibmacro{journal}
  \setunit*{\addspace}
  \iffieldundef{series}
    {}
    {\newunit
     \printfield{series}
     \setunit{\addspace}}
  \printfield{volume}
  \setunit{\addspace}
  \usebibmacro{issue+date}
  \setunit{\addcomma\space}
  \printfield{number}
  \setunit{\addcolon\space}
  \usebibmacro{issue}
  \setunit{\addcomma\space}
  \printfield{eid}
  \newunit}

\newtheoremstyle{mystyle}
  {}
  {}
  {\itshape}
  {}
  {\bfseries}
  {.}
  { }
  {\thmname{#1}\thmnumber{ #2}\thmnote{ (#3)}}

\theoremstyle{mystyle}
\newtheorem{theorem}{Theorem}[section]
\newtheorem{lemma}[theorem]{Lemma}
\newtheorem{corollary}[theorem]{Corollary}
\newtheorem{proposition}[theorem]{Proposition}

\theoremstyle{definition}
\newtheorem{definition}[theorem]{Definition}

\theoremstyle{remark}
\newtheorem{remark}[theorem]{Remark}

\newcommand{\CP}{\mathbb{CP}}

\title{Small undecidable groups and unrecognizable 4-manifolds}

\author{Marc Kegel}
\address{Universidad de Sevilla, Dpto.\ de Álgebra,
Avda.\ Reina Mercedes s/n,
41012 Sevilla, Spain}
\email{\href{mailto:kegelmarc87@gmail.com}{kegelmarc87@gmail.com}}

\author{Shana Yunsheng Li}
\address{Dept.~of Mathematics, University of Illinois Urbana-Champaign, Urbana, IL 61801, USA}
\email{\href{mailto:yl202@illinois.edu}{yl202@illinois.edu}}

\author{Qiuyu Ren}
\address{Department of Mathematics, Stanford University, Stanford, CA 94305, USA}
\email{\href{mailto:qren18@stanford.edu}{qren18@stanford.edu}}

\begin{document}

\begin{abstract}
We construct a $3$-generator $9$-relator group with unsolvable word problem. We use the group to construct two fixed-size Adian--Rabin families of group presentations, one with $4$ generators and $11$ relators, and another with $2$ generators and $10$ relators. As a consequence, $\#_7(S^2\times S^2)$ is topologically unrecognizable and $\#_9(S^2\times S^2)$ is smoothly unrecognizable. These algebraic and topological results improve the previous best known bounds by Borisov, Tancer, and Gordon.

The construction of the group builds upon an example of Borisov and uses additional HNN extensions and Tietze eliminations to reduce the size of the presentation. We also provide a machine-checked Lean~4 formalization of the algebraic results.
\end{abstract}

\maketitle

\section{Introduction}

Many undecidability results in topology are proved by showing that a hypothetical algorithm for a topological decision problem would yield an algorithm for a known undecidable problem in group theory. This approach, pioneered by Markov~\cite{Markov1958}, has led to numerous undecidable problems in topology. These group-theoretic decision problems are also of intrinsic interest in combinatorial group theory. A natural quantitative question is how small a finite presentation can be while still exhibiting algorithmically undecidable behavior.

\subsection{Algebra}
The \textit{word problem} for finitely presented groups, originating in the work of Dehn \cite{Dehn1911}, is one of the classical decision problems in group theory. A fundamental negative result was obtained by Novikov \cite{Novikov1952,Novikov1955} and independently by Boone \cite{Boone1959}, who constructed finitely presented groups with unsolvable word problem. Subsequent work sought increasingly economical examples: Boone \cite{Boone1959} obtained an example with 32 defining relators, Collins reduced this to 18 \cite{Collins1966} and subsequently to 14 \cite{Collins1969}, and Borisov \cite{Borisov1969} obtained an example with 12 relators. The latter was the smallest number of relators previously known for a finitely presented group with unsolvable word problem. Our first result improves upon this.

\begin{theorem}\label{thm:word_problem}
There exists a $3$-generator $9$-relator group with unsolvable word problem.
\end{theorem}

The \textit{triviality problem} for finitely presented groups is to determine, from a finite presentation, whether the presented group is trivial. Following Novikov's proof of the undecidability of the word problem \cite{Novikov1955}, Adian proved that the triviality problem is undecidable, as are many other recognition problems for finitely presented groups \cite{Adian1955,Adian1957}. Rabin independently obtained a general version of this result \cite{Rabin1958}, now known as the Adian--Rabin theorem.

The undecidability of the triviality problem can be quantified using the notion of Adian--Rabin families.
\begin{definition}\label{def:adian-rabin}
A recursively enumerable family of finite group presentations $\mathcal{P}$ is called an \emph{Adian--Rabin family} if there is no algorithm that takes as input $P\in\mathcal{P}$ and decides whether or not the group presented by $P$ is trivial.
\end{definition}

By the work of Adian~\cite{Adian1955,Adian1957} and Rabin~\cite{Rabin1958}, the set of all finite group presentations forms an Adian--Rabin family. To measure the ``smallness'' of Adian--Rabin families, one may ask for such families with additional restrictions, for example, with a fixed (small) number of relators, or a fixed deficiency\footnote{Here, following~\cite{Tancer2023}, the deficiency is defined as the number of relators minus the number of generators; this is the negative of the usual group-theoretic convention.}. A longstanding open problem asks whether there exists an Adian--Rabin family consisting entirely of balanced presentations. The best previously known constructions are due to Tancer~\cite{Tancer2023}, who constructed an Adian--Rabin family of deficiency $9$, and Gordon~\cite{Gordon2022}, who constructed one with $13$ relators. Our next results improve upon both bounds.

\begin{theorem}\label{thm:adian-rabin_deficiency}
There exists an Adian--Rabin family of $4$-generator, $11$-relator group presentations.
\end{theorem}

\begin{theorem}\label{thm:adian-rabin_relators}
There exists an Adian--Rabin family of $2$-generator, $10$-relator group presentations.
\end{theorem}

\subsection{Topology}
We say that a closed PL $n$-manifold $M$ is \emph{PL-unrecognizable} if there is no algorithm that, given a PL triangulation of a closed manifold $X$, decides whether or not $X$ is PL-homeomorphic to $M$. Similarly, $M$ is \emph{topologically unrecognizable} if there is no algorithm deciding whether or not $X$ is homeomorphic to $M$. In dimensions $n\le 6$, every PL $n$-manifold admits a compatible smooth structure, unique up to diffeomorphism~\cite{Milnor_overview}. Consequently, PL-unrecognizability and smooth unrecognizability are equivalent in these dimensions.

By Perelman's proof of the Geometrization Conjecture~\cite{Perelman2002,Perelman2003Surgery,Perelman2003Extinction}, every closed manifold of dimension at most three is recognizable~\cite{Kuperberg2019}. In contrast, the fundamental work of Markov~\cite{Markov1958} and Novikov (see \cite{Volodin1974,Chernavsky2006}) shows that there are unrecognizable manifolds in every dimension at least four. In fact, in dimensions greater than four, every nonempty closed PL manifold is PL-unrecognizable \cite{Chernavsky2006}.

In dimension four, the situation is more mysterious. It is a major open question whether $S^4$ is recognizable; see, for example, Problem~4.53 in the K3 problem list \cite{K3}. Markov's construction implies that the existence of an Adian--Rabin family with $r$ relators yields the smooth and topological unrecognizability of $\#_r(S^2\times S^2)$ (see, for example, \cite{friedl2026unsolvability}); the existence of such an $r$ was first shown in \cite{Adian1955}. Subsequent work reduced the number of $S^2\times S^2$ factors. In particular, Tancer \cite{Tancer2023} observed that a deficiency-$d$ Adian--Rabin family yields the topological unrecognizability of $\#_d(S^2\times S^2)$ as a consequence of Freedman's theorem \cite{Freedman1982}, and Gordon \cite{Gordon2022} observed that an $r$-relator Adian--Rabin family yields the smooth unrecognizability of $\#_{r-1}(S^2\times S^2)$ (although he only stated this in the topological category). Using their constructions of Adian--Rabin families mentioned above, they thereby showed the topological unrecognizability of $\#_9(S^2\times S^2)$ and the smooth unrecognizability of $\#_{12}(S^2\times S^2)$, respectively. Our Theorems~\ref{thm:adian-rabin_deficiency} and~\ref{thm:adian-rabin_relators} yield the following improvements.

\begin{corollary}
$\#_n(S^2\times S^2)$ is topologically unrecognizable for $n\ge7$ and smoothly unrecognizable for $n\ge9$.\qed
\end{corollary}

\begin{remark}
    Using analogous arguments, one can also show that $(\#_a\CP^2)\#(\#_b\overline{\CP^2})$ is topologically unrecognizable for $\min(a,b)\ge7$ and smoothly unrecognizable for $\min(a,b)\ge9$.
\end{remark}

\begin{remark}
Cameron Gordon informed us that he has independently obtained improvements on the smallest known Adian--Rabin families in the spirit of Theorems~\ref{thm:adian-rabin_deficiency} and~\ref{thm:adian-rabin_relators}, based on a modification of Borisov's group in \cite{Borisov1969}. Using these improvements, he was able to show both the smooth and the topological unrecognizability of $\#_8(S^2\times S^2)$. This will appear in an upcoming paper.
\end{remark}

\subsection{Constructions of the Adian--Rabin families}
\label{subsec:adian_rabin}
The group $G$ in Theorem~\ref{thm:word_problem} will satisfy the following extra properties.
\begin{proposition}\label{prop:extra_properties}
The group $G$ in Theorem~\ref{thm:word_problem} admits a $3$-generator $9$-relator presentation satisfying the following additional properties:
\begin{enumerate}[(1)]
\item There is a generator that normally generates $G$.
\item There is a generator that is trivial in the abelianization of $G$.
\end{enumerate}
\end{proposition}

Now we can combine the work of Gordon~\cite{Gordon2022} and Tancer~\cite{Tancer2023} with Theorem~\ref{thm:word_problem} and Proposition~\ref{prop:extra_properties} to deduce Theorems~\ref{thm:adian-rabin_deficiency} and~\ref{thm:adian-rabin_relators}.

\begin{proof}[Proof of Theorems~\ref{thm:adian-rabin_deficiency} and~\ref{thm:adian-rabin_relators}]
Theorem~\ref{thm:adian-rabin_relators} is a direct consequence of Gordon ~\cite[Lemma~2.1]{Gordon2022} applied to our group $G$ in Theorem~\ref{thm:word_problem}, noting that condition (2.1) in \cite{Gordon2022} follows from Proposition~\ref{prop:extra_properties}(2). 
    
Theorem~\ref{thm:adian-rabin_deficiency} follows from the construction in \cite[Theorem~9]{Tancer2023}. More precisely, in our setup, Tancer's argument builds a $6$-generator $13$-relator Adian--Rabin family from our group $G$ and the normal generator $r_1'$ provided by Proposition~\ref{prop:extra_properties}(1). However, GPT suggested that two additional Tietze eliminations can be performed to cancel two pairs of generators and relators. In the notation of \cite[(1)--(4)]{Tancer2023}, one first replaces the generator $\alpha$ by a new generator $\delta:=\gamma\alpha^{-1}$. Then one may cancel the generator $\gamma$ with relation (1), and the generator $r_1'$ with relation (4).
\end{proof}

\begin{remark}\label{rem:tancer}
As pointed out by GPT and confirmed to us by Martin Tancer, two corrections appear to be necessary in \cite{Tancer2023} concerning the construction of Adian--Rabin families from groups with unsolvable word problem.
\begin{enumerate}
\item The word $w$ in equation (3) of \cite{Tancer2023} should be replaced by $[w,\beta]$, as was done in Miller \cite{Miller1992}, because $\alpha^{-3}w\alpha^3$ cannot be a basis element of a free group if $w$ is torsion, breaking Tancer's argument.
\item Applying \cite[Remark~10]{Tancer2023} to the group $G$ would yield a $2$-generator $9$-relator Adian--Rabin family, beating both Theorems~\ref{thm:adian-rabin_deficiency} and~\ref{thm:adian-rabin_relators}. However, the remark seems to be unjustified.
\end{enumerate}
\end{remark}

\subsection{Proof outline of Theorem~\ref{thm:word_problem}}

We briefly discuss the proof of Theorem~\ref{thm:word_problem}, which uses a standard strategy in combinatorial group theory and
builds directly on the ideas in Borisov's construction~\cite{Borisov1969}.
We start with Borisov's group $\Gamma_4$, which has unsolvable word
problem~\cite{Borisov1969}. We then construct in Section~\ref{sec:construction} two HNN extensions
\begin{equation*}
    \Gamma_4\hookrightarrow\Gamma_5\hookrightarrow\Gamma_6.
\end{equation*}
For the first extension, we identify two suitable subgroups of
$\Gamma_4$ and an isomorphism between them; for the second, we do the
same for two subgroups of $\Gamma_5$. 
The main technical step is to prove that these subgroups have precisely the claimed presentations, with no additional relations. This is carried out in Section~\ref{sec:proof_lemmaG4-6}. Once this is established, the embeddings follow from the standard theory of HNN extensions, in particular, Britton's lemma; see, for example, \cite[Theorem~32(ii) and Proposition~34]{Cohen1989}.

Since $\Gamma_4$ embeds into $\Gamma_6$ and has unsolvable word problem,
$\Gamma_6$ also has unsolvable word problem. Finally, we simplify the
presentation of $\Gamma_6$: seven relators are redundant, and six
generator--relator pairs can be eliminated by Tietze transformations.
This leaves a presentation with three generators and nine relators,
proving Theorem~\ref{thm:word_problem}.

\subsection{Note on AI use}

The main construction in this paper was essentially developed by GPT-5.6 Sol Pro. Starting from Borisov's group with unsolvable word problem, the model suggested the two additional HNN extensions used in Sections~\ref{sec:construction} and~\ref{sec:proof_lemmaG4-6} and the subsequent Tietze transformations leading to the $3$-generator, $9$-relator presentation. It also suggested the two additional Tietze eliminations used in the proof of Theorem~\ref{thm:adian-rabin_deficiency} and drew our attention to the issues discussed in Remark~\ref{rem:tancer}.

We independently checked and reorganized these arguments and verified the literature references. We then used Codex extensively to develop the Lean~4 formalization as explained below. The authors have reviewed the final mathematical argument and take responsibility for its correctness. This paper does not contain any text directly generated by AI.

\subsection{Lean autoformalization}
With the assistance of GPT-5.6 Sol in Codex, we formalized in Lean~4 the proofs of Theorems~\ref{thm:word_problem},~\ref{thm:adian-rabin_deficiency}, and~\ref{thm:adian-rabin_relators}, together with all their dependencies, including Matiyasevich's $2$-letter $3$-relator semigroup/Thue system with unsolvable word problem \cite{Matiyasevich1995} and Borisov's construction \cite{Borisov1969}. The complete formalization is available in the accompanying GitHub repository \cite{Kegel2026AdianRabin}. We briefly summarize the development process.
\begin{enumerate}
\item We first used GPT-5.6 Sol Pro to produce an informal proof of Theorem~\ref{thm:word_problem}, which was recorded in a companion PDF.
\item We asked Codex to formalize the statements of the three main theorems. After some simplification, these occupy fewer than 80 lines of Lean code. We checked the resulting statements manually.
\item We then provided Codex with the informal proof and asked it to identify the main dependencies and construct the overall Lean proof skeleton. At this intermediate stage, five project-specific assumptions were represented by temporary Lean axioms. The most substantial ones concerned an undecidable $2$-generator $3$-relator semigroup/Thue system, initially attributed to Matiyasevich's 1967 construction \cite{Matiyasevich1967}, Borisov's unsolvability reduction from groups to semigroups/Thue systems, and the injectivity of the two further compressions from Borisov's group $\Gamma_4$.
\item For each temporary axiom, we asked Codex to identify the relevant literature, collected the sources in a dedicated reference directory, and then asked it to formalize the required arguments. This discharged all but the Thue-system input. Codex traced the route through Matiyasevich's 1967 construction back through earlier work to Novikov's undecidability theorem \cite{Novikov1955}, deciding that a direct formalization of that route would be disproportionately large.
\item To remove the remaining axiom, we followed a different route. Starting from a fixed Turing machine with an undecidable halting problem, Codex formalized Post's machine-to-Thue construction \cite{Post1947} and then Matiyasevich's compression to a binary three-rule Thue system \cite{Matiyasevich1995}. This completed the proof of the required undecidability statement.
\item We prompted Codex to find the improvement of \cite[Theorem~9]{Tancer2023} that we used in the proof of Theorem~\ref{thm:adian-rabin_deficiency} above. We asked Codex to formalize the reductions to Theorems~\ref{thm:adian-rabin_deficiency} and~\ref{thm:adian-rabin_relators} following Tancer \cite{Tancer2023} (with the improvement above) and Gordon \cite{Gordon2022}.
\item After the proof was complete, we asked Codex to audit the development for project-specific axioms and \texttt{sorry} declarations, remove redundant arguments, simplify proofs, and reorganize the repository. The complete proof contains approximately 28,000 lines of Lean code. A transitive axiom audit of the three main theorems reports only Lean's standard principles \texttt{propext}, \texttt{Classical.choice}, and
\texttt{Quot.sound}.
\item Finally, we verified that the formalization is accepted by the Lean Comparator \cite{Boving2025Comparator}.
\end{enumerate}

\subsection{Acknowledgements}
This work started during the 2026 Trisectors Workshop at Western Washington University. We thank the organizers of the workshop for fostering a stimulating environment for productive discussion and research. We warmly thank Nathan Dunfield for his substantial help with this project. We thank Cameron Gordon for his careful reading of a first draft of this paper and for many valuable comments. We thank Martin Tancer for clarifications regarding his work~\cite{Tancer2023}. QR thanks Shurui Liu for helpful discussions on AI tools and autoformalization.

MK was supported by a Ram\'on y Cajal grant \mbox{(RYC2023-043251-I)} and PID2024-157173\-NB-I00 funded by MCIN/AEI/10.13039/501100011033, by ESF+, and by FEDER, EU; and by a VII Plan Propio de Investigación y Transferencia (SOL2025-36103) of the University of Sevilla. SYL was partially supported by the US National Science Foundation grant DMS-2303572. This research was conducted during the period when QR served as a Clay Research Fellow.

\section{The construction}
\label{sec:construction}

We start with a finitely presented semigroup $\Pi$ of the form 
\begin{equation*}
    \Pi = \langle s_1, s_2\mid F_i\sim E_i,\ i=1,2,3\rangle
\end{equation*}
such that:
\begin{enumerate}[(a)]
    \item There exists a positive word $P$ over $\{s_1,s_2\}$ such that there is no algorithm to decide, given a word $Q$, whether or not $Q\sim P$ in $\Pi$;
    \item The words $E_i$, $F_i$ ($i=1,2,3$) and $P$ all contain both of the generators $s_1$ and $s_2$;
    \item In a group containing $s_1$ and $s_2$, if $s_1=s_2^{-1}$ and $F_2=E_2$, then $s_1=s_2=1$.
\end{enumerate}
Such a $\Pi$ can be obtained from Matiyasevich's construction \cite[eq. (6)]{Matiyasevich1995} by applying a faithful encoding to it; see Appendix~\ref{appen:encoding}.

Consider the following sequence of groups and natural homomorphisms $$\Gamma_1\longrightarrow\Gamma_2\longrightarrow\cdots \longrightarrow\Gamma_6,$$ where each $\Gamma_j$ is given by the relators enclosed in the square brackets in \eqref{eq:tower} and the generators $s_1, s_2, c, d, e, k, t, y, z$ occurring in the corresponding relators. The subscripts $i$ and $\beta$ vary in $\{1, 2, 3\}$ and $\{1, 2\}$, respectively.

\begin{equation}
\label{eq:tower}
\begin{gathered}
\Gamma_6\!\left[
\begin{array}{c}
\Gamma_5\!
\left[
\begin{array}{c}
\Gamma_4\!
\left[
\begin{array}{c}
\Gamma_{3}\!\left[
  \begin{array}{c}
    \Gamma_{2}\!\left[
      \begin{array}{c}
        \Gamma_{1}\!\left[\, d^{4}s_{\beta}=s_{\beta}d,\quad es_{\beta}=s_{\beta}e^{4} \,\right]\\[3pt]
        s_{\beta}c=cs_{\beta}\\[3pt]
        d^{i}F_{i}e^{i}c=cd^{i}E_{i}e^{i}
      \end{array}
    \right]\\[3pt]
    ct=tc,\quad dt=td
  \end{array}
\right]\\[3pt]
ck=kc,\quad ek=ke,\quad (P^{-1}tP)k=k(P^{-1}tP)
\end{array}\right]\\[3pt]
dye=y,\quad ey=yd,\quad cy=yc,\quad 
ty=yk,\quad s_1ys_2=y
\end{array}\right]\\[6pt]
dz=zt,\quad tz=zc,\quad cz=zy
\end{array}\right]
\end{gathered}
\end{equation}

For example, $\Gamma_2 = \langle s_1,s_2, c,d,e\mid  d^{4}s_{\beta}=s_{\beta}d,\ es_{\beta}=s_{\beta}e^4, s_{\beta}c=cs_{\beta},\ d^iF_ie^ic=cd^iE_ie^i\rangle$. 

We claim that these natural homomorphisms arise from HNN extensions and are therefore injective. The subsequence $\langle d,e\rangle\le\Gamma_1\le\Gamma_2\le\Gamma_3\le\Gamma_4$ of embeddings was established by Borisov~\cite{Borisov1969}. He also showed the following.

\begin{theorem}[Borisov]\label{thm:borisov}
    If $Q$ is a positive word over $\{s_1,s_2\}$, then $Q\sim P$ in $\Pi$ if and only if $(Q^{-1}tQ)k = k(Q^{-1}tQ)$ in $\Gamma_4$. Consequently, $\Gamma_4$ has unsolvable word problem by property (a) of $\Pi$. \qed
\end{theorem}

Thus, it remains to establish the following lemma, whose proof is deferred to Section~\ref{sec:proof_lemmaG4-6}.

\begin{lemma}
\label{lem:G4-6}
    $\Gamma_4\leq\Gamma_5\leq\Gamma_6$.
\end{lemma}

It follows from Lemma~\ref{lem:G4-6} that $G\coloneqq \Gamma_6$ also has unsolvable word problem. In the rest of this section, we prove the remaining claims in Theorem~\ref{thm:word_problem} and Proposition~\ref{prop:extra_properties}.

\begin{proposition}
\label{pr:G6-rel}
    $\Gamma_6$ admits a presentation with $3$ generators and $9$ relators.
\end{proposition}

\begin{proof}
$\Gamma_6$ is a group with $9$ generators and $22$ relators by construction. Among the relators, the following seven
\begin{equation*}
    d^4s_2=s_2d,\quad es_2=s_2e^4,\quad s_2c=c s_2,\quad ck=kc,\quad ek=ke, \quad cy=yc, \quad ct=tc
\end{equation*}
are redundant: The first five follow from the $y$-conjugates of 
$$es_1=s_1e^4,\quad d^4s_1=s_1d, \quad s_1c=cs_1,\quad ct=tc,\quad dt=td,$$
respectively, and then the last two follow from the $z$-conjugates of $ct=tc$ and $dt=td$, respectively. Next, we perform six Tietze eliminations to cancel six additional generator-relator pairs by rewriting 
\begin{equation}\label{eq:substitution}
t\coloneqq z^{-1}dz,\ \ c\coloneqq z^{-1}tz,\ \ y\coloneqq z^{-1}cz,\ \ e\coloneqq ydy^{-1},\ \ s_2\coloneqq y^{-1}s_1^{-1}y,\ \ k\coloneqq y^{-1}ty.
\end{equation}
We thus obtain a $3$-generator $9$-relator presentation 
\begin{align}\label{eq:G_presentation}
&\Gamma_6\cong\langle d,s_1,z\mid d^4s_1=s_1d,es_1=s_1e^4,s_1c=cs_1,\nonumber\\&d^iF_ie^ic=cd^iE_ie^i(i=1,2,3),dt=td,(P^{-1}tP)k=k(P^{-1}tP),dye=y\rangle,
\end{align}
where $t,c,y,e,s_2,k$ are words in $d,s_1,z$ determined iteratively by \eqref{eq:substitution}.
\end{proof}

Next, we show that the group $\Gamma_6$ satisfies the properties in Proposition~\ref{prop:extra_properties}. More precisely, we have:

\begin{proposition}\label{prop:extra_properties_detailed}
Using the presentation \eqref{eq:G_presentation}, $\Gamma_6$ is normally generated by $z$, and $d$ is trivial in the abelianization $(\Gamma_6)_{ab}$ of $\Gamma_6$.
\end{proposition}

\begin{proof}
Setting $z=1$, \eqref{eq:substitution} gives $d=t=c=y=e=k$. The relation $dye=y$ gives $d^2=1$. Noting that $s_1$ commutes with $d$ by $s_1c=cs_1$, the relation $d^4s_1=s_1d$ gives $d^3=1$, so $d=1$. Finally, the relation $s_2=y^{-1}s_1^{-1}y$ in \eqref{eq:substitution} gives $s_1=s_2^{-1}$, and the relation $d^2F_2e^2c=cd^2E_2e^2$ gives $E_2=F_2$, so property (c) of $\Pi$ implies $s_1=s_2=1$. This proves the first part of the proposition.

Abelianizing \eqref{eq:substitution} gives $d=t=c=y=e=k$ and $s_1=-s_2$. The relations $d^4s_1=s_1d$ and $dye=y$ then give $3d=0$ and $2d=0$, respectively. Hence $d=0$ in $(\Gamma_6)_{\mathrm{ab}}$.
\end{proof}

Assuming Lemma~\ref{lem:G4-6}, we can give a proof of Theorem~\ref{thm:word_problem}.

\begin{proof}[Proof of Theorem~\ref{thm:word_problem} and Proposition~\ref{prop:extra_properties}]
    By Lemma~\ref{lem:G4-6}, $\Gamma_4$ is a subgroup of $\Gamma_6$. Borisov's theorem~\ref{thm:borisov} implies that $\Gamma_4$ has unsolvable word problem, and hence so does $\Gamma_6$. Proposition~\ref{pr:G6-rel} proves that $\Gamma_6$ admits a presentation with $3$ generators and $9$ relators, and the additional properties claimed in Proposition~\ref{prop:extra_properties} follow from Proposition~\ref{prop:extra_properties_detailed}.   
\end{proof}

\section{Proof of Lemma~\ref{lem:G4-6}}\label{sec:proof_lemmaG4-6}

Thus, it remains to prove Lemma~\ref{lem:G4-6}. We divide the proof into two parts. For $\Gamma_4\leq \Gamma_5$, consider the subgroups $\mathcal{A}$ and $\mathcal{B}$ of $\Gamma_4$ generated by $\{s_1,c,d,e,t\}$ and $\{s_2,c,d,e,k\}$ respectively. It suffices to show the following.

\begin{lemma}
\label{lem:G4-5}
$\mathcal{A}$ and $\mathcal{B}$ have presentations
\begin{equation*}
    \mathcal{A}\cong \langle s_1,c,d,e,t\mid d^4s_1=s_1d,\  es_1=s_1e^4,\  s_1c=cs_1,\  ct=tc, \  dt=td\rangle,
\end{equation*}
\begin{equation*}
    \mathcal{B}\cong \langle s_2,c,d,e,k\mid d^4s_2=s_2d,\  es_2=s_2e^4,\  s_2c=cs_2,\  ck=kc, \  ek=ke\rangle.
\end{equation*}
\end{lemma}

Hence the assignment
\begin{equation*}
    d\mapsto e^{-1}, \quad e\mapsto d, \quad c\mapsto c, \quad t\mapsto k, \quad s_1\mapsto s_2^{-1}
\end{equation*}
respects the relations and thus defines an isomorphism of subgroups $\theta\colon \mathcal{A}\to \mathcal{B}$, and $\Gamma_5$ is the HNN extension of $\Gamma_4$ relative to $\theta$. 

For $\Gamma_5\leq \Gamma_6$, consider the subgroups $\mathcal{C}$ and $\mathcal{D}$ of $\Gamma_5$ generated by $\{c,d,t\}$ and $\{c,t,y\}$ respectively. It suffices to prove the following.

\begin{lemma}
\label{lem:G5-6}
    $\mathcal{C}$ and $\mathcal{D}$ have presentations
    \begin{equation*}
        \mathcal{C}\cong \langle c,d,t\mid ct=tc,\  dt=td\rangle,
    \end{equation*}
    \begin{equation*}
        \mathcal{D}\cong \langle c,t,y\mid ct=tc,\  cy=yc\rangle.
    \end{equation*}
\end{lemma}

Hence the assignment 
\begin{equation*}
    c\mapsto y,\quad d\mapsto t,\quad t\mapsto c
\end{equation*}
defines an isomorphism from $\mathcal{C}$ to $\mathcal{D}$, and $\Gamma_6$ is the HNN extension of $\Gamma_5$ relative to this isomorphism. 

We will appeal to lemmas in Appendix~\ref{appen:HNN} for the proofs of Lemmas~\ref{lem:G4-5}~and~\ref{lem:G5-6}.

\subsection{Proof of Lemma~\ref{lem:G4-5}}

To complete the analysis of $\mathcal{A}$ and $\mathcal{B}$ in $\Gamma_4$, we use the following facts about the extension $\langle d, e\rangle \leq \Gamma_1$ proved by Borisov \cite[Assertions IV \& V]{Borisov1969}. 

\begin{proposition}[Borisov, Assertion IV]
    The morphisms from free groups $\langle A_1,\dots, A_5\rangle$ and $\langle B_1,\dots, B_5\rangle$ to $\Gamma_1$,
    \begin{equation*}
        \varphi_A\colon \langle A_1,\dots, A_5\rangle\to \Gamma_1,\quad \varphi_B\colon \langle B_1,\dots, B_5\rangle\to \Gamma_1
    \end{equation*}
    defined by 
    \begin{equation*}
        \varphi_A(A_\beta) = s_\beta, \quad \varphi_A(A_{2+i})= d^iF_ie^i,\quad \varphi_B(B_\beta) = s_\beta, \quad \varphi_B(B_{2+i}) = d^iE_ie^i
    \end{equation*}
    are monomorphisms. \qed
\end{proposition}

\begin{proposition}[Borisov, Assertion V]
\label{pr:Borisov-V}
    For a reduced word in $\langle A_1,\dots, A_5\rangle$ (resp. $\langle B_1,\dots, B_5\rangle$), its expansion in $s_1,s_2,d,e$ obtained by directly applying $\varphi_A$ (resp. $\varphi_B$) is HNN-reduced with respect to the multi-HNN extension $\langle d, e\rangle \leq \Gamma_1$. \qed
\end{proposition}

For $\beta = 1,2$, let $S_\beta$ denote the subgroup of $\Gamma_1$ generated by $\{s_\beta, d,e\}$. We now prove the following.

\begin{lemma}
\label{lem:ABintS}
    For $\beta =1,2$, 
    \begin{equation}
    \label{eq:ABintS}
        \varphi_A(\langle A_1,\dots,A_5\rangle)\cap S_\beta = \langle s_\beta\rangle,\quad  \varphi_B(\langle B_1,\dots,B_5\rangle)\cap S_\beta = \langle s_\beta\rangle,
    \end{equation}
    and
    \begin{equation}
    \label{eq:ABintde}
         \varphi_A(\langle A_1,\dots,A_5\rangle)\cap \langle d,e\rangle = 1,\quad \varphi_B(\langle B_1,\dots,B_5\rangle)\cap\langle d,e\rangle = 1. 
    \end{equation}
\end{lemma}
\begin{proof}
    It suffices to prove the statements for $\varphi_A$ and $\beta = 1$; the proofs for the other cases are identical. 
    
    For \eqref{eq:ABintS}, the inclusion $\langle s_1\rangle \leq \varphi_A(\langle A_1,\dots,A_5\rangle)\cap S_1$ is obvious. 
    Conversely, let $g\in  \varphi_A(\langle A_1,\dots,A_5\rangle)\cap S_1$. There exists a reduced word in $\langle A_1,\dots,A_5\rangle$ whose expansion via $\varphi_A$ is a word $W$ in $s_1,s_2,d,e$ representing $g$. By Proposition~\ref{pr:Borisov-V}, $W$ is HNN-reduced with respect to $\langle d,e\rangle \leq \Gamma_1$. 

    On the other hand, since $g\in S_1$, there exists a word $W'$ in $s_1, d,e$ representing $g$ that is HNN-reduced with respect to $\langle d,e\rangle \leq S_1$. By definition, $W'$ is also HNN-reduced with respect to $\langle d,e\rangle \leq \Gamma_1$ since it does not contain the letter $s_2$. By Lemma~\ref{lem:multi-Britton}, $W$ likewise contains no occurrence of $s_2$. By Property~(b) of $\Pi$, each of $\varphi_A(A_2),\ldots,\varphi_A(A_5)$ contains the letter $s_2$, whereas $\varphi_A(A_1)=s_1$. Hence the reduced word whose expansion is $W$ can involve only $A_1$, and therefore $W$ is a power of $s_1$.

    For \eqref{eq:ABintde} it suffices to show that $\langle d,e\rangle \cap \langle s_1\rangle = 1$. This follows from the observation that in the abelianization of $\Gamma_1$, $s_1$ contributes to a free $\mathbb{Z}$ summand while both $d$ and $e$ have finite order. 
\end{proof}

Since $\Gamma_1\leq \Gamma_2$ is the HNN extension relative to the group isomorphism from $\varphi_A(\langle A_1,\dots, A_5\rangle)$ to $\varphi_B(\langle B_1,\dots, B_5\rangle)$ induced by $A_j\mapsto B_j$, taking $A = \varphi_A(\langle A_1,\dots, A_5\rangle)$ and $B=\langle d, e\rangle$ in Lemma~\ref{lem:restrictHNN}, we obtain the following corollary from \eqref{eq:ABintde}.

\begin{corollary}
\label{cor:cde}
    The subgroup of $\Gamma_2$ generated by $\{c,d,e\}$ is free on $\{c,d,e\}$ and its intersection with $\Gamma_1$ is exactly $\langle d,e\rangle$.\qed
\end{corollary}

Similarly, taking $A = \varphi_A(\langle A_1,\dots, A_5\rangle)$ and $B = S_\beta$ in Lemma~\ref{lem:restrictHNN}, we deduce the following corollary from \eqref{eq:ABintS}.

\begin{corollary}
\label{cor:scde}
    For $\beta = 1,2$, the subgroup of $\Gamma_2$ generated by $\{s_\beta, c,d,e\}$ has presentation
    \begin{equation}
    \label{eq:scde}
        \langle s_\beta, c,d,e\mid d^4s_\beta = s_\beta d,\  es_\beta = s_\beta e^4,\ s_\beta c = cs_\beta\rangle, 
    \end{equation}
    and its intersection with $\Gamma_1$ is exactly $S_\beta$. \qed
\end{corollary}

Write $p\coloneqq P^{-1}tP$ and observe that
\begin{equation}
\label{eq:G2-3}
    \Gamma_3 = \langle \Gamma_2, p\mid ap=pa,\ a\in P^{-1}\langle c,d\rangle P\rangle.
\end{equation}
Let $\Delta$ be the subgroup of $\Gamma_3$ generated by $\{c, e, p\}$. Taking $B$ in Corollary~\ref{cor:centralHNN} as the subgroup $\langle c, e \rangle$ of $\Gamma_2$ (note that $\langle c,e\rangle$ is free due to Corollary~\ref{cor:cde}), we see that
\begin{equation*}
\Delta\cap \Gamma_2 =\langle c,e\rangle.
\end{equation*} 
Furthermore, we observe that $\Gamma_3\leq \Gamma_4$ is the HNN extension relative to $1_\Delta$. 
We are now ready to prove Lemma~\ref{lem:G4-5}.

\begin{proof}[Proof of Lemma~\ref{lem:G4-5}]
    The desired presentation of $\mathcal{A}$ is obtained immediately by applying Corollary~\ref{cor:centralHNN} with $B$ being the group in \eqref{eq:scde} for $\beta =1$.

    For $\mathcal{B}$, let $B$ be the group in \eqref{eq:scde} for $\beta =2$. Then $\Delta\cap B = \langle c,e\rangle$ since
    \begin{equation*}
        \langle c,e\rangle \leq \Delta\cap B \leq \Delta\cap \Gamma_2 = \langle c,e\rangle.
    \end{equation*}
    Applying Corollary~\ref{cor:centralHNN} with this $B$ (and $G=\Gamma_3\leq\Gamma_4=H$) therefore gives the desired presentation of $\mathcal{B}$. 
\end{proof}

\subsection{Proof of Lemma~\ref{lem:G5-6}}

Let $D$ and $D_t$ be the subgroups of $\Gamma_3$ generated by $\{c,d,e\}$ and $\{c,d,e,t\}$ respectively; we have seen from Corollary~\ref{cor:cde} that $D$ is freely generated by $\{c,d,e\}$. Since $\Gamma_2\leq \Gamma_3$ is the HNN extension relative to the identity on the subgroup generated by $\{c,d\}$, applying Corollary~\ref{cor:centralHNN} with $B = D$ gives that
\begin{equation*}
    D_t\cong \langle c,d,e,t\mid ct=tc,\ dt=td\rangle.
\end{equation*}
Restricting to $\mathcal{C}$, we obtain the desired presentation
\begin{equation*}
    \mathcal{C} \cong \langle c,d,t\mid ct=tc,\ dt=td\rangle.
\end{equation*}
For the presentation of $\mathcal{D}$, recall that $D\cap \Gamma_1 = \langle d,e\rangle$ from Corollary~\ref{cor:cde}. Let $S$ be the subgroup of $\Gamma_1$ generated by $\{s_1,s_2\}$. Then $P\in S$, and consequently $P\not\in D$ since $D$ intersects with only the torsion summands of $\Gamma_1$ in the abelianization. Hence we may apply Lemma~\ref{lem:conjHNN} with this $P$ and $D$ along with $K=\langle c,e\rangle$ and $G=\Gamma_2$, obtaining
\begin{equation*}
    \Delta\cap D_t=\langle c,e\rangle. 
\end{equation*}
Now, since $\Gamma_3\leq \Gamma_4$ is the HNN extension relative to $1_\Delta$, applying Corollary~\ref{cor:centralHNN} with $B = D_t$ shows that the subgroup of $\Gamma_4$ generated by $\{c,d,e,t,k\}$ admits the presentation
\begin{equation*}
 \langle c,d,e,k,t\mid ck=kc,\ ek=ke,\ ct=tc,\ dt=td \rangle.
\end{equation*}
Since trivializing $d$ and $e$ gives a left inverse of the natural map
\begin{equation*}
    \langle c,k,t \mid ck=kc,\  ct=tc\rangle \to \langle c,d,e,k, t\mid ck=kc,\ ek=ke,\ ct=tc,\ dt=td  \rangle,
\end{equation*}
it follows that the subgroup $B_k$ of $\Gamma_4$ generated by $\{c,k,t\}$ is presented by
\begin{equation*}
    B_k\cong \langle c, k, t \mid ck=kc, \ ct=tc\rangle.
\end{equation*}
We are now only one proposition away from obtaining the desired presentation of $\mathcal{D}$.

\begin{proposition}
\label{pr:kintAB}
    $B_k\cap \mathcal{A}$ and $B_k\cap \mathcal{B}$ are the subgroups of $\Gamma_4$ generated by $\{c,t\}$ and $\{c,k\}$ respectively. 
\end{proposition}
With Proposition~\ref{pr:kintAB}, applying Lemma~\ref{lem:restrictHNN} with $B = B_k$ and $A = \mathcal{A}$, we obtain that the subgroup of $\Gamma_5$ generated by $\{c, k, t, y\}$ has presentation
\begin{equation*}
    \langle c,k,t,y\mid ck=kc, \ ct=tc, \ ty=yk, \ cy=yc\rangle,
\end{equation*}
where $k$ is redundant because $k=y^{-1}ty$; eliminating it gives
\begin{equation*}
    \mathcal{D} \cong \langle c,t,y\mid ct=tc,\ cy=yc\rangle
\end{equation*}
as desired. 

\begin{proof}[Proof of Proposition~\ref{pr:kintAB}]
    Let $C_t$ and $C_k$ be the subgroups of $\Gamma_4$ generated by $\{c,t\}$ and $\{c,k\}$ respectively. 
    For $B_k\cap \mathcal{A}$, it suffices to show that $B_k\cap \Gamma_3 = C_t$ since $C_t\leq \mathcal{A}\leq \Gamma_3$. Since $\Gamma_3\leq \Gamma_4$ is the HNN extension relative to $1_\Delta$, this follows immediately from Corollary~\ref{cor:centralHNN} with $B=C_t$.

    For $B_k\cap \mathcal{B}$, using the notation of Lemma~\ref{lem:intCentHNN} and set $G= \Gamma_3$, $A=\Delta$, $L= C_t$, $C=\langle c\rangle$. Let $M$ be the subgroup of $\Gamma_3$ generated by $\{s_2, c,d,e\}$. It suffices to show that 
    \begin{equation*}
        C_t\cap \Delta M\Delta =\langle c\rangle.
    \end{equation*}
    It is clear that $\langle c\rangle \subset C_t\cap \Delta M \Delta$. For the converse, note that an element $g\in C_t\cap\Delta M \Delta$ is represented by a word of the form
    \begin{equation}
    \label{eq:DMD-form}
        g= h_1p^{\epsilon_1}\cdots p^{\epsilon_n}h_{n+1} \cdot m \cdot g_1p^{\mu_1}\cdots p^{\mu_{n'}}g_{n'+1},
    \end{equation}
    where $m \in M$, $h_j, g_j \in \langle c,e\rangle$, and $\epsilon_j, \mu_j\in \{\pm 1\}$. Recall from \eqref{eq:G2-3} that $\Gamma_2\leq \Gamma_3$ is the HNN extension relative to $1_{P^{-1}\langle c,d\rangle P}$, hence reducing $p$-pinches preserves the form in \eqref{eq:DMD-form}. In particular, $g$ admits an HNN-reduced expression of the form
    \begin{equation}
    \label{eq:p-reduced}
        g= m_1p^{\varepsilon_1}\cdots p^{\varepsilon_{n''}}m_{{n''}+1},
    \end{equation}
    where $\varepsilon_j\in \{\pm1\}$ and $m_j\in M$ since $\langle c,e\rangle \subset M$. Substituting $p= P^{-1}tP$ gives
    \begin{equation*}
        g= (m_1P^{-1})t^{\varepsilon_1} (Pm_2P^{-1})t^{\varepsilon_2}\cdots t^{\varepsilon_{n''}} (Pm_{n''+1}),
    \end{equation*}
    which is HNN-reduced with respect to the extension $\Gamma_2\leq \Gamma_3$ obtained by adjoining $t$, since a $t$-pinch would imply that $Pm_jP^{-1} \in \langle c,d\rangle $ for some $j$ with $\varepsilon_{j-1}=-\varepsilon_j$, which implies that $m_j\in P^{-1}\langle c,d\rangle P$, giving a $p$-pinch in \eqref{eq:p-reduced}. 

    On the other hand, since $g\in C_t$, it admits the following reduced form in terms of the $t$-extension:
    \begin{equation*}
        g = \gamma_1 t^{\varepsilon_1}\cdots t^{\varepsilon_{n''}}\gamma_{{n''}+1},
    \end{equation*}
    where $\gamma_j\in \langle c\rangle$. Hence
    \begin{equation*}
        (m_1P^{-1})t^{\varepsilon_1} (Pm_2P^{-1})t^{\varepsilon_2}\cdots t^{\varepsilon_{n''}} (Pm_{n''+1})\gamma_{n''+1}^{-1} t^{-\varepsilon_{n''}}\cdots t^{-\varepsilon_1}\gamma_1^{-1} = 1. 
    \end{equation*}
    If $n''\neq 0$, Lemma~\ref{lem:multi-Britton} implies that $t^{\varepsilon_{n''}} (Pm_{n''+1})\gamma_{{n''}+1}^{-1} t^{-\varepsilon_{n''}}$ is a $t$-pinch, hence 
    \begin{equation*}
        (Pm_{n''+1})\gamma_{{n''}+1}^{-1}\in \langle c,d\rangle,
    \end{equation*}
    which implies that $P\in M$. However, by the property (b) of $\Pi$, this is not possible since, recalling that $M\cap \Gamma_1 = S_2$ from Corollary~\ref{cor:scde}, $M$ intersects with only the free summand given by $s_2$ in the abelianization of $\Gamma_1$, while $P$ projects nontrivially to both the free summands given by $s_1$ and $s_2$. Therefore $n''=0$, which gives $g\in \langle c\rangle$ as desired. 
\end{proof}

\appendix
\section{A faithful encoding of Matiyasevich's construction}
\label{appen:encoding}

In \cite[eq.~(6)]{Matiyasevich1995}, a semigroup $\Sigma$ of the form
\begin{equation*}
    \Sigma = \langle x, y\mid \mathcal{F}_i\sim \mathcal{E}_i,\ i=1,2,3 \rangle
\end{equation*}
was constructed such that
\begin{enumerate}
    \item[(i)] There exists a positive word $\mathcal{P}$ in $x,y$ such that there is no algorithm to decide, given a word $\mathcal{G}$, whether or not $\mathcal{G}\sim \mathcal{P}$ in $\Sigma$;
    \item[(ii)] $\mathcal{F}_1=xxyxy$, $\mathcal{E}_1=yxx$,  $\mathcal{F}_2=xxyy$ and $\mathcal{E}_2 = yxx$.
\end{enumerate}
It is clear that $\Sigma$ satisfies the properties (a) and (c) listed at the beginning of Section~\ref{sec:construction}. The property (b), however, is not immediately clear from the literature, as $\mathcal{F}_3$, $\mathcal{E}_3$ and $\mathcal{P}$ were not explicitly stated. Therefore we apply a faithful encoding to $\Sigma$ to avoid having to determine these words explicitly. 

Consider the homomorphism $\eta$ from the free semigroup $\langle x,y\rangle$ to the free semigroup $\langle s_1,s_2\rangle$ defined by
\begin{equation*}
    \eta(x) = s_1s_2^2s_1, \ \eta(y) = s_1s_2^3s_1.
\end{equation*}
Put $F_i\coloneqq \eta(\mathcal{F}_i)$, $E_i \coloneqq  \eta(\mathcal{E}_i)$ and $P\coloneqq \eta(\mathcal{P})$, and define
\begin{equation*}
    \Pi \coloneqq   \langle s_1, s_2\mid F_i\sim E_i,\ i=1,2,3\rangle.
\end{equation*}
It is straightforward to see that $\Pi$ satisfies both properties (b) and (c). For property (a), it suffices to show the following:

\begin{proposition}
    For any two positive words $U$ and $V$ in $x$ and $y$, $U\sim V$ in $\Sigma$ if and only if $\eta(U)\sim \eta(V)$ in $\Pi$. 
\end{proposition}

\begin{proof}
    The forward implication is immediate. For the converse, observe that $\eta$ is injective and that every word in its image has a unique factorization into copies of $\eta(x)$ and $\eta(y)$. It suffices to show that no occurrence of $F_i$ or $E_i$ can begin in the interior of a factor $\eta(x)$ or $\eta(y)$, so that we can translate replacements $F_i\leftrightarrow E_i$ in $\Pi$ back into replacements in $\Sigma$. Indeed, both $F_i$ and $E_i$ begin with $s_1s_2$, which occurs only at the very beginning of each $\eta(x)$ and $\eta(y)$ for an element in the image of $\eta$ since it must be of the form
    \begin{equation*}
        s_1s_2^{j_1}s_1^2s_2^{j_2}s_1^2\cdots s_2^{j_n}s_1,\quad j_k\in \{2,3\}.\qedhere
    \end{equation*}
\end{proof}

\section{Lemmas about HNN extensions}
\label{appen:HNN}

Let $G$ be a group, $A \leq G$ be a subgroup of $G$ and $\varphi\colon A \to G$ be a monomorphism. Let $\Lambda$ be a finite index set and $H = \langle G, u_\alpha\mid au_\alpha=u_\alpha\varphi(a),\forall a\in A \quad (\alpha\in \Lambda)\rangle$ be a (multi-)HNN extension of $G$ relative to $\varphi$. A word in $G, u_\alpha$ is \textit{HNN-reduced} with respect to $G\leq H$ if it contains no \textit{pinches}, i.e., subwords of the form $u_\alpha^{-1}au_\alpha$ with $a\in A$ or $u_\alpha bu_\alpha^{-1}$ with $b\in \varphi(A)$. By starting with an arbitrary word in $G, u_\alpha$ and replacing $u_\alpha^{-1}au_\alpha$ and $u_\alpha b u_\alpha^{-1}$ with $\varphi(a)$ and $\varphi^{-1}(b)$ respectively, one sees that every element of $H$ has a representative word that is HNN-reduced with respect to $G\leq H$.

The following two facts are standard \cite[Theorem 32 (ii) \& Proposition 34]{Cohen1989}.

\begin{lemma}[Britton's lemma]
\label{lem:multi-Britton}
    For an element $h\in H =\langle G, u_\alpha\mid au_\alpha=u_\alpha\varphi(a),\forall a\in A\quad (\alpha\in \Lambda)\rangle$, there exists a unique sequence
    \begin{equation*}
        (\alpha_1, \epsilon_1), \dots, (\alpha_{n}, \epsilon_{n}),
    \end{equation*}
    where $\epsilon_j\in \{\pm 1\}$ for all $j=1,\dots, n$, such that every HNN-reduced word representing $h$ is of the form
    \begin{equation*}
        g_1u_{\alpha_1}^{\epsilon_1}g_2u_{\alpha_2}^{\epsilon_2}\cdots u_{\alpha_n}^{\epsilon_n}g_{n+1},
    \end{equation*}
    where $g_j\in G$ for all $j = 1,\dots, n+1$.\qed
\end{lemma}

\begin{lemma}
\label{lem:restrictHNN}
    Let $B\le G$. If $\varphi(B\cap A) = B\cap \varphi(A)$, then for any subset $\Theta\subset \Lambda$, the subgroup of $H$ generated by $B\cup\{u_\alpha\mid \alpha\in \Theta\}$ is presented by
    $$\langle B,u_\alpha \mid bu_\alpha=u_\alpha\varphi(b),\ \forall b\in B\cap A \quad (\alpha \in \Theta)\rangle,$$ 
    and its intersection with $G$ is exactly $B$.\qed
\end{lemma}

Taking $\varphi = 1_A$, the assumption $\varphi(B\cap A)=B\cap \varphi(A)$ is automatic. Thus, we obtain the following corollary.

\begin{corollary}
\label{cor:centralHNN}
    For any $B\leq G$, the subgroup of $H=\langle G, u \mid au=ua, \forall a\in A\rangle$ generated by $B\cup \{u\}$ is presented by
    $$\langle B, u\mid bu=ub,\ \forall b\in B\cap A \rangle,$$ 
    and its intersection with $G$ is $B$. \qed
\end{corollary}

Let us fix $H = \langle G, u \mid au=ua, \forall a\in A\rangle$ below and denote by $\langle B, u\rangle$ the subgroup generated by $B\cup \{u\}$ in $H$. 
The proofs of the following two statements are routine applications of Lemma~\ref{lem:multi-Britton}.

\begin{lemma}
\label{lem:conjHNN}
    If $K, A\leq D\leq G$ and $P\in G\setminus D$, then $\langle K, P^{-1}uP\rangle \cap \langle D,u\rangle = K$. 
\end{lemma}
\begin{proof}
    It is clear that $K\subset \langle K, P^{-1}uP\rangle \cap \langle D,u\rangle$. For the converse, set $p\coloneqq P^{-1}uP$. Then
    \begin{equation*}
        H = \langle G, p \mid bp = pb,\ \forall b\in P^{-1}AP\rangle. 
    \end{equation*}
    Let $g\in \langle K,p\rangle \cap \langle D, u \rangle$. 
    Since the HNN extension is relative to the identity map, expanding $g$ in $K,p$ and applying pinch reduction gives the following HNN-reduced form:
    \begin{equation}
    \label{eq:kp-reduced}
        g = k_1p^{\epsilon_1}\cdots k_n p^{\epsilon_n}k_{n+1},
    \end{equation}
    where $\epsilon_j\in \{\pm 1\}$ and $ k_j\in K$ with $k_j\not\in P^{-1}AP$ whenever $2\le j\le n$ and $\epsilon_{j-1}=-\epsilon_j$. Substituting $p= P^{-1}uP$ back gives
    \begin{equation*}
        g = (k_1P^{-1}) u^{\epsilon_1} (Pk_2P^{-1}) u^{\epsilon_2}\cdots u^{\epsilon_n} (P k_{n+1}),
    \end{equation*}
    which is also HNN-reduced in terms of the $u$-extension, since a $u$-pinch would imply $k_j\in P^{-1}AP$ for some $j$ with $\epsilon_{j-1}=-\epsilon_j$. 

    On the other hand, expanding $g$ in $D,u$ gives the following HNN-reduced form:
    \begin{equation*}
        g = d_1u^{\epsilon_1}\cdots u^{\epsilon_n}d_{n+1},
    \end{equation*}
    where $d_j \in D$. Hence
    \begin{equation*}
        (k_1P^{-1}) u^{\epsilon_1} (Pk_2P^{-1}) u^{\epsilon_2}\cdots u^{\epsilon_n} (P k_{n+1}) d_{n+1}^{-1} u^{-\epsilon_n}\cdots u^{-\epsilon_1} d_1^{-1} = 1.
    \end{equation*}
    If $n\neq 0$, $ u^{\epsilon_n} (P k_{n+1}) d_{n+1}^{-1} u^{-\epsilon_n}$ must be a $u$-pinch by Lemma~\ref{lem:multi-Britton}, which gives
    \begin{equation*}
        Pk_{n+1}d_{n+1}^{-1}\in A.
    \end{equation*}
    Since $k_{n+1}, d_{n+1}\in D$, this implies that $P\in D$, contradicting the assumption. Therefore $n =0$, so $g\in K$ from \eqref{eq:kp-reduced}.
\end{proof}

\begin{lemma}
\label{lem:intCentHNN}
    Let $A, L, M\leq G$. If $C\leq A\cap L \cap M$ satisfies $C = L\cap AMA$, where $AMA$ denotes the indicated product set, then $\langle L,u\rangle\cap \langle M, u\rangle = \langle C, u\rangle$. 
\end{lemma}
\begin{proof}
    It is clear that $\langle C, u\rangle\subset \langle L,u\rangle\cap \langle M, u\rangle$. For the converse, let $g\in \langle L,u\rangle\cap \langle M, u\rangle$. Since the HNN extension is relative to the identity map, expanding $g$ in $L,u$ and $M,u$ respectively gives HNN-reduced forms 
    \begin{equation*}
        g =  l_1u^{\epsilon_1}\cdots u^{\epsilon_n}l_{n+1},
    \end{equation*}
    where $\epsilon_j\in \{\pm 1\}$ and $l_j\in L$, and
    \begin{equation*}
        g =  m_1u^{\epsilon_1}\cdots u^{\epsilon_n}m_{n+1},
    \end{equation*}
    where $m_j\in M$. Hence
    \begin{equation*}
        l_1u^{\epsilon_1}\cdots u^{\epsilon_n}l_{n+1} m_{n+1}^{-1} u^{-\epsilon_n}\cdots u^{-\epsilon_1} m_1^{-1} = 1,
    \end{equation*}
    and by Lemma~\ref{lem:multi-Britton}
    \begin{equation*}
        l_{n+1}m_{n+1}^{-1}\in A,
    \end{equation*}
    hence $l_{n+1}\in L\cap AMA = C$. It follows that $l_j\in C$ for all $j$ by repeatedly applying $u$-pinch reductions, concluding that $g\in \langle C, u\rangle$.
\end{proof}

\printbibliography

\end{document}